\documentclass[12pt,a4paper]{amsart}
\usepackage[utf8]{inputenc}
\usepackage[T1]{fontenc}
\usepackage{amssymb,amscd}
\usepackage{amsfonts}
\usepackage[top=35mm, bottom=35mm, left=30mm, right=30mm]{geometry}
\usepackage[colorlinks=true,citecolor=blue]{hyperref}
\usepackage{mathptmx}
\usepackage{eucal}
\usepackage{graphicx}
\usepackage{mathrsfs}
\usepackage{amssymb}
\usepackage{amsmath}
\usepackage{amsthm}
\usepackage{amscd}

\usepackage{xcolor}
\usepackage{cite}

\theoremstyle{plain}
\newtheorem{main}{Theorem}

\newtheorem{theorem}{Theorem}[section]
\newtheorem{proposition}[theorem]{Proposition}
\newtheorem{lemma}[theorem]{Lemma}
\newtheorem{corollary}[theorem]{Corollary}
\newtheorem{definition}[theorem]{Definition}
\newtheorem{remark}[theorem]{Remark}
\newtheorem{question}[theorem]{Question}
\newtheorem{example}[theorem]{Example}

\newcommand{\M}{\mathcal M}
\newcommand{\Me}{\mathcal M^e}
\newcommand{\B}{\mathcal B}

\makeatletter

\newcommand{\Rmnum}[1]{\expandafter\@slowromancap\romannumeral #1@}
\makeatother

\begin{document}
\title{Empirical variational principles for preimage entropies}
\author{Tao Wang}
\address{MOE-LCSM, School of Mathematics and Statistics, Hunan Normal University, Changsha, Hunan 410081, P. R. China}
\email{twang@hunnu.edu.cn}

\author{Yi Yang}
\address{School of Mathematics (Zhuhai), Sun Yat-sen University, Zhuhai 519082, P. R. China}
\email{yangy699@mail.sysu.edu.cn}



\begin{abstract}
Preimage entropy measures the complexity generated by the inverse-image structure of a non-invertible dynamical system. For a continuous map $f:X\to X$ on a compact metric space, Hurley's pointwise topological preimage entropies $h_m(f)$ and $h_p(f)$ are natural invariants measuring the complexity of preimage sets. The question of whether they admit unconditional variational principles in terms of suitable measure-theoretic counterparts remains open. In this paper we resolve it by using empirical metric preimage entropies $h^*_{m,\mu}(f)$ and $h^*_{p,\mu}(f)$, defined by restricting preimage fibers to orbit segments whose empirical measures are close to a prescribed invariant measure $\mu$. We prove the variational principles
$$
h_m(f)=\sup_{\mu\in\mathcal M_f(X)}h^*_{m,\mu}(f),
\qquad
h_p(f)=\sup_{\mu\in\mathcal M_f(X)}h^*_{p,\mu}(f)
$$
for every continuous map on a compact metric space. We also show that, in general, the set of all invariant measures in these formulas cannot be replaced by the set of ergodic invariant measures. We then compare the empirical entropies with the pointwise metric preimage entropy $h_{m,\mu}(f)$. For every ergodic invariant measure $\mu$, we prove $h^*_{p,\mu}(f)\ge h_{m,\mu}(f)$. Moreover, if $f$ has uniform separation of preimages, then for every ergodic invariant measure $\mu$,
$$
h^*_{m,\mu}(f)=h^*_{p,\mu}(f)=h_{m,\mu}(f).
$$
Examples show that $h_{m,\mu}(f)$ is not comparable with the empirical quantities in general. We further introduce a resolving-partition property, weaker than uniform separation of preimages, under which the variational principle for $h_m(f)$ and $h_{m,\mu}(f)$ holds. Finally, we establish corresponding variational principles for preimage pressure and give an example showing that uniform separation of preimages does not imply forward expansiveness.	
\end{abstract}


\maketitle

\tableofcontents

\section{Introduction}

Entropy is one of the central invariants in both topological and measure-theoretic dynamics. For a homeomorphism, the forward and backward orbit structures of a point are symmetric in many respects. For a non-invertible continuous map, however, the inverse images of a point may branch and spread in a complicated way. Classical topological entropy captures forward orbit separation, but it does not by itself reveal the complexity arising from the inverse structure. This observation led to a class of entropy-like quantities, now usually called \emph{preimage entropies}, whose purpose is to measure the non-invertible part of a dynamical system by examining iterated preimage sets.

Let $(X,d)$ be a compact metric space and let $f:X\to X$ be a continuous map. The first systematic ideas in this direction go back to Langevin and Przytycki \cite{LangevinPrzytycki1992}. Hurley \cite{Hurley1995} later introduced and studied pointwise topological preimage entropies for continuous maps on compact metric spaces. In current notation, two key quantities are
\[
 h_m(f)=\lim_{\varepsilon\to0}\limsup_{n\to\infty}
       \frac1n\log \sup_{x\in X}s(n,\varepsilon,f^{-n}x)
\]
and
\[
 h_p(f)=\sup_{x\in X}\lim_{\varepsilon\to0}\limsup_{n\to\infty}
       \frac1n\log s(n,\varepsilon,f^{-n}x),
\]
where $s(n,\varepsilon,Z)$ denotes the largest cardinality of an $(n,\varepsilon)$-separated subset of $Z$. Clearly, $h_p(f)\leq h_m(f)$. These quantities vanish for homeomorphisms but can be positive for non-invertible maps. Hurley also linked preimage entropy to branch entropy and topological entropy. This line of work was subsequently extended by Nitecki and Przytycki \cite{NiteckiPrzytycki1999}, Fiebig, Fiebig, and Nitecki \cite{FiebigFiebigNitecki2003}, and Nitecki \cite{Nitecki2003}. In particular, it was shown that if $f$ has uniform separation of preimages, then $h_p(f)=h_m(f)$, and if $f$ is forward expansive, then $h_p(f)=h_m(f)=h_{\rm top}(f)$. Clearly, forward expansiveness implies uniform separation of preimages (see Section \ref{subsec:two conditions} for the definitions). However, to the best of our knowledge, no example of a non-invertible continuous map that has uniform separation of preimages but is not forward expansive has been reported in the literature. In this paper, we provide such an example in the final section; see Section \ref{sec:an example}.

A second line of development aimed to construct measure-theoretic counterparts of topological preimage entropy and to formulate variational principles. This raises a natural and important question:

\begin{question}\label{que:main question}
Can one introduce measure-theoretic counterparts of $h_m(f)$ or $h_p(f)$ for an $f$-invariant measure $\mu$, and establish a variational principle linking topological preimage entropy to its measure-theoretic version?
\end{question}

A major advance was made by Wu and Zhu \cite{WuZhu2021}, who introduced a pointwise metric preimage entropy
\[
 h_{m,\mu}(f)=\sup_{\alpha}\limsup_{n\to\infty}
 \frac1n H_\mu(\alpha_0^{n-1}\mid f^{-n}\mathcal B),
\]
where $\alpha$ ranges over all finite measurable partitions. This definition is closely aligned with Hurley's $h_m(f)$ and serves as a natural measure-theoretic counterpart of $h_m(f)$, since the conditioning sigma-algebra $f^{-n}\mathcal B$ corresponds to the $n$-step preimage fiber. Under the condition of uniform separation of preimages, Wu and Zhu proved the variational principle
\[
 h_m(f)=\sup_{\mu\in\mathcal M_f(X)} h_{m,\mu}(f)
 =\sup_{\mu\in\mathcal M_f^e(X)} h_{m,\mu}(f),
\]
where $\mathcal M_f(X)$ and $\mathcal M_f^e(X)$ denote the sets of all $f$-invariant and ergodic $f$-invariant Borel probability measures on $X$, respectively. They also related $h_{m,\mu}(f)$ to folding entropy and stable entropy, established upper semi-continuity results, and obtained a Shannon-McMillan-Breiman type theorem. Related results and further developments can be found in \cite{ChengNewhouse2005, WuZhu2022, Wang2024}.

A natural question then arises: does a variational principle hold for $h_m(f)$ and $h_{m,\mu}(f)$ without any additional assumptions? Recently, Shi, Yan, and Zeng \cite{ShiYanZeng2023} showed that there exists a continuous map $f$ on a compact metric space $X$ for which the following strict inequality holds:
\[
 \sup_{\mu\in\mathcal M_f(X)} h_{m,\mu}(f) < h_m(f).
\]
Thus, a variational principle between $h_m(f)$ and $h_{m,\mu}(f)$ fails without extra conditions. Consequently, the answer to Question~\ref{que:main question} remains open.

In this paper, we show that notions of measure-theoretic preimage entropy based on empirical measures can be used to establish variational principles for $h_m(f)$ and $h_p(f)$. These notions were introduced and studied by Cheng, Li, and Wu \cite{ChengLiWu2023}, following the approach of Pfister and Sullivan \cite{PfisterSullivan2007}. We call these quantities the \emph{empirical metric preimage entropies} and denote them by $h_{m,\mu}^*(f)$ and $h_{p,\mu}^*(f)$, respectively; see Section~\ref{sec:pre} for precise definitions. In general, we have $h_{p,\mu}^*(f)\leq h_{m,\mu}^*(f)$. These empirical preimage entropies record which invariant measures arise as limits of the orbit distributions of large preimage separated sets. In summary, we are able to prove variational principles for $h_m(f)$ and $h_p(f)$ for general continuous maps without any additional assumptions, thereby resolving Question~\ref{que:main question}. Our first main result is the following:

\begin{main}\label{thm:variational principle}
For every continuous map $f:X\to X$ on a compact metric space,
\[
 h_m(f)=\sup_{\mu\in\mathcal{M}_f(X)}h_{m,\mu}^{*}(f),\qquad
 h_p(f)=\sup_{\mu\in\mathcal{M}_f(X)}h_{p,\mu}^{*}(f).
\]
\end{main}

However, we will show that the set of $f$-invariant measures $\mathcal{M}_f(X)$ cannot be replaced by the set of ergodic $f$-invariant measures $\mathcal{M}_f^e(X)$ in Theorem~\ref{thm:variational principle}. That is,

\begin{main}\label{thm:ergodic measure}
There exist a compact metric space $(X,d)$ and a continuous map $f:X\to X$ such that
\[
 h_p(f)>\sup_{\mu\in\mathcal{M}_f^e(X)}h_{m,\mu}^{*}(f).
\]
\end{main}

These results in turn lead naturally to the following question:

\begin{question}\label{que:relation}
What is the relationship between the empirical metric preimage entropies $h_{m,\mu}^*(f)$, $h_{p,\mu}^*(f)$ and the metric preimage entropy $h_{m,\mu}(f)$?
\end{question}

We address this question in the present paper by obtaining the following result:

\begin{main}\label{thm:relations for ergodic case}
Let $f:X\to X$ be continuous. If $\mu\in\mathcal{M}_f^e(X)$, then
\[
 h_{p,\mu}^{*}(f)\ge h_{m,\mu}(f).
\]
Furthermore, if $f$ has uniform separation of preimages and $\mu\in\mathcal{M}_f^e(X)$, then
\[
 h_{m,\mu}^{*}(f)=h_{p,\mu}^{*}(f)= h_{m,\mu}(f).
\]
\end{main}

Moreover, without any extra conditions, we show that there is no definite relation among $h_{m,\mu}^*(f)$, $h_{p,\mu}^{*}(f)$, and $h_{m,\mu}(f)$; some interesting examples are given to illustrate this.

\begin{main}\label{thm:relations for general case}
There exist a compact metric space $(X,d)$, a continuous map $f:X\to X$, and a measure $\mu\in\mathcal{M}_f(X)$ such that
\[
h_{m,\mu}^{*}(f) < h_{m,\mu}(f).
\]
Moreover, there exist a compact metric space $(X,d)$, a continuous map $f:X\to X$, and a measure $\mu\in\mathcal{M}_f(X)$ such that
\[
h_{p,\mu}^{*}(f) > h_{m,\mu}(f).
\]
\end{main}

On the other hand, certain classical interval maps---such as the full tent map---do not satisfy the uniform separation of preimages condition. This leads us to ask whether the uniform separation of preimages can be weakened so that it covers such classical examples while still ensuring that the variational principle between $h_m(f)$ and $h_{m,\mu}(f)$ holds. Some progress in this direction is also obtained in this paper. More precisely, we introduce a notion of \emph{resolving-partition property} and prove the following:

\begin{main}\label{thm:resolving partition}
If the continuous map $f:X\to X$ has the resolving-partition property, then
\[
h_p(f)=h_m(f)=\sup_{\mu\in\mathcal{M}_f(X)}h_{m,\mu}(f)
       =\sup_{\mu\in\mathcal{M}_f^e(X)}h_{m,\mu}(f).
\]
\end{main}

Furthermore, the pressure version of Theorem~\ref{thm:variational principle} is also discussed.

The paper is organized as follows. Section \ref{sec:pre} recalls the necessary definitions.
Section \ref{sec:variational principle} proves the empirical variational principles and constructs the
counterexample showing that the supremum cannot be restricted to ergodic
measures. Section \ref{sec:relations} compares the empirical metric preimage entropies with the
pointwise metric preimage entropy $h_{m,\mu}(f)$. Section \ref{sec:resolving-partition property} introduces the
resolving-partition property and proves the variational principle
under this condition. Section \ref{sec:variational principle for preimage pressure} proves the corresponding variational
principles for preimage pressure. The final section gives an example showing that uniform separation
of preimages does not imply forward expansiveness.

\section{Preliminaries}\label{sec:pre}

Throughout this section, $(X,d)$ is a compact metric space and $f$ is a continuous self-map on $X$.

\subsection{Topological preimage entropy}

For $n\in\mathbb{N}$ and $\varepsilon>0$, define the Bowen metric and Bowen ball by
\[
d_n(x,y)=\max_{0\le i\le n-1}d(f^ix,f^iy),\qquad
B_n(x,\varepsilon)=\{y\in X:d_n(x,y)<\varepsilon\}.
\]
A set $E\subset X$ is \emph{$(n,\varepsilon)$-spanning} if $X=\bigcup_{x\in E}B_n(x,\varepsilon)$,
and \emph{$(n,\varepsilon)$-separated} if $d_n(x,y)\ge\varepsilon$ for all distinct $x,y\in E$.
Denote by $s_f(n,\varepsilon,E)$ the maximum cardinality of an $(n,\varepsilon)$-separated subset of $E$. 
When the map $f$ is clear from the context, we simply write $s(n,\varepsilon,E)$.

We consider two pointwise topological preimage entropies:
\[
h_m(f)=\lim_{\varepsilon\to0}\limsup_{n\to\infty}\frac1n\log\sup_{x\in X}s(n,\varepsilon,f^{-n}x),
\]
\[
h_p(f)=\sup_{x\in X}\lim_{\varepsilon\to0}\limsup_{n\to\infty}\frac1n\log s(n,\varepsilon,f^{-n}x).
\]

\subsection{Measure-theoretic preimage entropy}

Let $\mathcal{B}$ be the Borel $\sigma$-algebra on $X$. Denote by $\mathcal{M}(X)$ the set of all Borel probability measures on $X$, and by $\mathcal{M}_f(X)$ and $\mathcal{M}_f^e(X)$ the subsets of $f$-invariant and ergodic $f$-invariant measures, respectively. For $\mu\in\mathcal{M}(X)$, we work with the $\mu$-completion of $\mathcal{B}$, still denoted by $\mathcal{B}$. 

A \emph{partition} of $X$ is a collection of disjoint subsets whose union is $X$. For a partition $\xi$, let $\xi(x)$ be the element containing $x$. Write $\xi\succeq\eta$ (or $\eta\preceq\xi$) if $\xi(x)\subset\eta(x)$ for all $x\in X$. The refinement of $\xi$ and $\eta$ is $\xi\vee\eta=\{A\cap B:A\in\xi,\,B\in\eta\}$.

Let $(X,\mu,\mathcal{A})$ be a standard probability space. A partition $\xi$ is \emph{measurable} if there exists a countable family $\{A_n\}\subset\mathcal{B}(\xi)$ separating almost every pair of atoms, where $\mathcal{B}(\xi)$ is the sub-$\sigma$-algebra of $\mathcal{A}$ consisting of unions of elements of $\xi$. Let $\varepsilon$ be the partition of $X$ into singletons; it is measurable and generates $\mathcal{B}$. In what follows, we identify a measurable partition with the $\sigma$-algebra it generates when writing conditional entropy. The \emph{canonical system of conditional measures} for $\mu$ and $\xi$ is a family $\{\mu_x^\xi:x\in X\}$ satisfying $\mu_x^\xi(\xi(x))=1$ for $\mu$-a.e.\ $x$, such that for every measurable $B\subset X$, the map $x\mapsto\mu_x^\xi(B)$ is $\mathcal{B}(\xi)$-measurable and
\[
\mu(B)=\int_X\mu_x^\xi(B)\,d\mu(x).
\]
Such a system exists and is essentially unique (see \cite{EinsiedlerWard2011,Rokhlin1962}).

For measurable partitions $\alpha$ and $\xi$ of $(X,\mu,\mathcal{A})$, define the \emph{information function}
\[
I_\mu(\alpha)(x)=-\log\mu(\alpha(x)),
\]
and the \emph{conditional information function}
\[
I_\mu(\alpha\mid\xi)(x)=-\log\mu_x^\xi(\alpha(x)).
\]
The \emph{entropy} and \emph{conditional entropy} are
\[
H_\mu(\alpha)=\int_X I_\mu(\alpha)(x)\,d\mu(x),\qquad
H_\mu(\alpha\mid\xi)=\int_X I_\mu(\alpha\mid\xi)(x)\,d\mu(x).
\]

For $\mu\in\mathcal{M}_f(X)$ and a finite measurable partition $\alpha$, let
\[
h_{m,\mu}(f,\alpha)=\limsup_{n\to\infty}\frac1n H_\mu\bigl(\alpha_0^{n-1}\mid f^{-n}\mathcal{B}\bigr).
\]
Then the \emph{pointwise metric preimage entropy} is defined by
\[
h_{m,\mu}(f)=\sup_\alpha h_{m,\mu}(f,\alpha),
\]
where $\alpha$ runs over all finite Borel partitions.
Another related quantity is the \emph{folding entropy}
\[
F_\mu(f)=H_\mu(\varepsilon\mid f^{-1}\varepsilon),
\]
introduced by Ruelle \cite{Ruelle1996}.

We now turn to another type of measure-theoretic preimage entropy, introduced and studied by Cheng, Li and Wu \cite{ChengLiWu2023}, following the approach of Pfister and Sullivan \cite{PfisterSullivan2007}. Unlike the conditional-entropy-based definition above, this entropy is defined using the empirical measures of orbits of points taken from preimage sets. Accordingly, we refer to it as the \emph{empirical preimage entropy}.
For $y\in X$, the empirical measure of the first $n$ iterates is
\[
\mathcal{E}_n(y)=\mathcal{E}_n^f(y)=\frac1n\sum_{j=0}^{n-1}\delta_{f^jy}\in\mathcal{M}(X).
\]
For an open set $G\subset\mathcal{M}(X)$, define
\[
P_n(x,G)=P_n^f(x,G)=f^{-n}x\cap\{y\in X:\mathcal{E}_n(y)\in G\}.
\]

\begin{definition}
Let $\mu\in\mathcal{M}_f(X)$. For $\varepsilon>0$ and a neighbourhood $G$ of $\mu$, define
\[
h_{m,\mu}^*(f,\varepsilon;G)=\limsup_{n\to\infty}\frac1n\log\sup_{x\in X}s\bigl(n,\varepsilon,P_n(x,G)\bigr),
\]
and
\[
h_{p,\mu}^*(f,x,\varepsilon;G)=\limsup_{n\to\infty}\frac1n\log s\bigl(n,\varepsilon,P_n(x,G)\bigr).
\]
Then set
\[
h_{m,\mu}^*(f,\varepsilon)=\inf_{G\ni\mu}h_{m,\mu}^*(f,\varepsilon;G),\qquad
h_{p,\mu}^*(f,x,\varepsilon)=\inf_{G\ni\mu}h_{p,\mu}^*(f,x,\varepsilon;G),
\]
where $G$ ranges over all neighbourhoods of $\mu$. We call
\[
h_{m,\mu}^*(f)=\lim_{\varepsilon\to0}h_{m,\mu}^*(f,\varepsilon),\qquad
h_{p,\mu}^*(f)=\sup_{x\in X}\lim_{\varepsilon\to0}h_{p,\mu}^*(f,x,\varepsilon)
\]
the \emph{$m$-type} and \emph{$p$-type empirical metric preimage entropy}, respectively.
\end{definition}

\subsection{Two conditions}\label{subsec:two conditions}
Let $f$ be a continuous self-map on a compact metric space $(X,d)$. We recall the definitions of uniform separation of preimages and forward expansiveness.
\begin{definition}
$f$ is said to have \emph{uniform separation of preimages} if there exists $\varepsilon_0>0$ such that
\[
d(x,y)\le\varepsilon_0,\; f(x)=f(y)\;\Longrightarrow\; x=y.
\]
Such $\varepsilon_0$ is called an \emph{exponent of separation} for $f$.
\end{definition}

\begin{definition}
$f$ is \emph{forward expansive} if there exists $\delta>0$ such that for every pair $x\neq y$ in $X$, there exists $n\ge0$ with $d(f^nx,f^ny)>\delta$. Equivalently,
\[
d(f^nx,f^ny)\le\delta\;\forall n\ge0\;\Longrightarrow\; x=y.
\]
Such $\delta$ is called a \emph{forward expansive constant} for $f$.
\end{definition}

Clearly, forward expansiveness implies uniform separation of preimages.

\section{The variational principle}\label{sec:variational principle}

\subsection{Proof of Theorem \ref{thm:variational principle}}

We begin by proving Theorem \ref{thm:variational principle}. To this end, we first establish two useful lemmas.

\begin{lemma}\label{lem:variational inequality1}
For every $\mu\in\M_f(X)$,
\[
 h_{m,\mu}^{*}(f)\le h_m(f).
\]
\end{lemma}

\begin{proof}
For any neighbourhood $G\ni\mu$ and any $x\in X$, we have
\[
 P_n(x,G)\subset f^{-n}x,
\]
which implies
\[
 s(n,\varepsilon,P_n(x,G))\le s(n,\varepsilon,f^{-n}x).
\]
The desired inequality follows directly.
\end{proof}

For fixed $\varepsilon>0$, write
\[
 h_m(f,\varepsilon)=\limsup_{n\to\infty}\frac1n\log\sup_{x\in X}s(n,\varepsilon,f^{-n}x).
\]

\begin{lemma}\label{lem:variational inequality2}
For every $\varepsilon>0$ and every $a<h_m(f,\varepsilon)$, there exists $\mu\in\M_f(X)$ such that
\[
h_{m,\mu}^{*}(f,\varepsilon)\ge a.
\]
\end{lemma}

\begin{proof}
Choose a sequence $n_i\to\infty$, points $x_i\in X$, and $(n_i,\varepsilon)$-separated sets
\[
 E_i\subset f^{-n_i}x_i
\]
with $\#E_i\ge e^{a n_i}$ for all sufficiently large $i$.

Let $D$ be a metric inducing the weak-star topology on $\mathcal{M}(X)$. Since $\mathcal{M}(X)$ is compact, for each $r\ge1$ we can cover $\mathcal{M}(X)$ by finitely many $D$-balls
\[
U_{r,1},\dots,U_{r,N_r}
\]
each of diameter less than $1/r$. For fixed $r$ and $i$, the empirical measures $\{\mathcal{E}_{n_i}(y):y\in E_i\}$ are contained in the union of these balls. By the pigeonhole principle, some ball contains at least $\#E_i/N_r$ of them.

Now proceed inductively. For $r=1$, by the pigeonhole principle we may select a ball $G_1$ from the finite collection $\{U_{1,1},\dots,U_{1,N_1}\}$ and an infinite subsequence $I_1\subset\mathbb{N}$ such that for every $i\in I_1$, at least $\#E_i/N_1$ points in $E_i$ have their empirical measure lying in $G_1$. Next, for $r=2$, choose a ball $G_2$ from $\{U_{2,1},\dots,U_{2,N_2}\}$ and an infinite subsequence $I_2\subset I_1$ with the same property. Continuing this process, we obtain balls $G_r\in\{U_{r,1},\dots,U_{r,N_r}\}$ and nested infinite subsequences
\[
I_1\supset I_2\supset I_3\supset\cdots
\]
such that for each $i\in I_r$, there exists a subset $E_i^{(r)}\subset E_i$ satisfying
\[
\#E_i^{(r)}\ge\frac{\#E_i}{N_r},\qquad
\mathcal{E}_{n_i}(y)\in G_r\ \text{for all}\ y\in E_i^{(r)},
\]
and $\operatorname{diam}G_r<1/r$.

Fix $r\in\mathbb{N}$. For any $i\in I_r$, pick $y_{i,r}\in E_i^{(r)}$ and set $\mu_i^r:=\mathcal{E}_{n_i}(y_{i,r})\in G_r$. Passing to a subsequence if necessary, we may assume that $\mu_i^r\to\mu\in\mathcal{M}(X)$ as $i\to\infty$ within $I_r$. It is well known that $\mu\in\mathcal{M}_f(X)$.

Let $G$ be any neighbourhood of $\mu$. Since $\mu_i^r\to\mu$ and $\operatorname{diam}G_r<1/r$, for all sufficiently large $r$ we have $G_r\subset G$. For any $i\in I_r$,
\[
 s\bigl(n_i,\varepsilon,P_{n_i}(x_i,G)\bigr)\ge \#E_i^{(r)}\ge \frac{\#E_i}{N_r}.
\]
Therefore
\[
\limsup_{n\to\infty}\frac1n\log\sup_{x\in X}s\bigl(n,\varepsilon,P_n(x,G)\bigr)\ge a.
\]
Taking the infimum over all $G\ni\mu$ gives $h_{m,\mu}^{*}(f,\varepsilon)\ge a$.
\end{proof}

Combining Lemmas \ref{lem:variational inequality1} and \ref{lem:variational inequality2} yields the first variational principle in Theorem \ref{thm:variational principle}. The proof of the second one is similar.

\begin{remark}
In the the section before the last one, we shall generalize the variational principles established in Theorem~\ref{thm:variational principle} to the setting of preimage pressure by means of a different approach. In this section, we prove Theorem~\ref{thm:variational principle} by explicitly constructing an invariant measure.
\end{remark}

\subsection{Proof of Theorem \ref{thm:ergodic measure}}

The purpose of this subsection is to give a compact metric system for which the $p$-type topological preimage entropy is positive, while all ergodic invariant measures have zero $m$-type empirical metric preimage entropy.
More precisely, we construct a compact metric space $X$ and a continuous map $F$ on $X$ such that
\[
 h_p(F)>0
 \quad\text{but}\quad
 h^*_{m,\nu}(F)=0\quad\text{for every }\nu\in \Me_F(X).
\]
Consequently,
\[
 h_p(F)>\sup_{\nu\in\Me_F(X)}h^*_{m,\nu}(F).
\]
Since always $h^*_{p,\nu}(F)\le h^*_{m,\nu}(F)$, this also implies
\[
h_m(F)\geq h_p(F)>\sup_{\nu\in\Me_F(X)}h^*_{m,\nu}(F)\geq\sup_{\nu\in\Me_F(X)}h^*_{p,\nu}(F).
\]
This proves Theorem \ref{thm:ergodic measure}.

\subsubsection{The space}
Let
\[
 \Sigma=\{0,1\}^{\mathbb Z}
\]
with the two-sided full shift $\sigma:\Sigma\to\Sigma$, $(\sigma\omega)_i=\omega_{i+1}$.
Let $H=\ell^2(\mathbb Z)$ with standard orthonormal basis $(e_i)_{i\in\mathbb Z}$.
Define $\Phi:\Sigma\to H$ by
\[
 \Phi(\omega)=\sum_{i\in\mathbb Z}2^{-|i|-3}\omega_i e_i.
\]
Then $\Phi$ is a continuous, one-to-one map.
Thus $Y=\Phi(\Sigma)$ is compact. Also, if $\omega_0\ne\eta_0$, then
\[
 \|\Phi(\omega)-\Phi(\eta)\|_H\ge 2^{-3}=\frac18.
\]
The shift induces a homeomorphism $S:Y\to Y$ by
\[
 S\Phi(\omega)=\Phi(\sigma\omega).
\]

Work in the Hilbert space
\[
 \mathbb H=\mathbb R^2\oplus H.
\]
We write points as $(a,t,u)$ and use the metric
\[
 D((a,t,u),(a',t',u'))=|a-a'|+|t-t'|+\|u-u'\|_H.
\]
Define
\[
 w(t)=1-t,\qquad 0\le t\le1,
\]
and
\[
 \Omega(t,\omega)=w(t)\Phi(\omega)=(1-t)\Phi(\omega).
\]
Thus $\Omega(1,\omega)=0$ for every $\omega$. At the level $t=1$, the symbolic fibre is
completely collapsed.

Define the limit part
\[
 L=\{(0,t,\Omega(t,\omega)):0\le t<1,\ \omega\in\Sigma\}\cup\{q\},
 \qquad q=(0,1,0).
\]
We call $L_t=\{(0,t,\Omega(t,\omega)):\ \omega\in\Sigma\}$ with $0\le t<1$ a visible limit layer. 
For each $m\ge1$, define the finite tower
\[
 T_m=\left\{\left(\frac1m,\frac jm,
        \Omega\left(\frac jm,\omega\right)\right):
        0\le j\le m,\ \omega\in\Sigma\right\}.
\]
The top point is
\[
 p_m=\left(\frac1m,1,0\right).
\]
Now set
\[
 X=L\cup\bigcup_{m\ge1}T_m\subset \mathbb H.
\]
We use the metric $D$ restricted to $X$.

\begin{lemma}
The space $X$ is compact.
\end{lemma}

\begin{proof}
Let $\{x_r\}_{r\geq1}\subset X$ be any sequence.
If infinitely many $x_r$ lie in $L$, then compactness of $[0,1]$ and compactness of
$Y=\Phi(\Sigma)$ give a convergent subsequence in $L$.
If infinitely many $x_r$ lie in a single fixed tower $T_m$, then a subsequence converges in $T_m$, because $T_m$ is a finite union of compact sets.
It remains to consider $x_r\in T_{m_r}$ with $m_r\to\infty$. Write
\[
 x_r=\left(\frac1{m_r},\frac{j_r}{m_r},
       \left(1-\frac{j_r}{m_r}\right)\Phi(\omega_r)\right).
\]
After passing to a subsequence, assume $j_r/m_r\to t\in[0,1]$. If $t<1$, then compactness of
$Y$ gives a further subsequence such that $\Phi(\sigma^{j_r}\omega_r)\to\Phi(\eta)$ for
some $\eta\in\Sigma$. Then
\[
 x_r\to (0,t,(1-t)\Phi(\eta))\in L.
\]
If $t=1$, then
\[
 \left\|\left(1-\frac{j_r}{m_r}\right)\Phi(\omega_r)\right\|_H
 \le \left(1-\frac{j_r}{m_r}\right)\sup_{\omega\in\Sigma}\|\Phi(\omega)\|_H\to0,
\]
so $x_r\to q$. Thus every sequence has a convergent subsequence with limit in $X$.
\end{proof}

\subsubsection{The map}

Define $F:X\to X$ as follows. On the limit part, put
\[
 F(0,t,\Omega(t,\omega))=(0,t,\Omega(t,\sigma\omega)),\qquad 0\le t<1,
\]
and
\[
 F(q)=q.
\]
On $T_m$, for $0\le j<m$, put
\[
 F\left(\frac1m,\frac jm,
        \Omega\left(\frac jm,\omega\right)\right)
 =
 \left(\frac1m,\frac{j+1}{m},
        \Omega\left(\frac{j+1}{m},\sigma\omega\right)\right).
\]
Finally, the top of every finite tower is sent to the common fixed point:
\[
 F(p_m)=q.
\]
This last rule is the key point. All finite towers eventually fall into the same point $q$.

\begin{lemma}
The map $F:X\to X$ is continuous.
\end{lemma}

\begin{proof}
Let $x_r\to x$ in $X$. If $x$ lies in a fixed finite tower $T_m$, then
$x_r\in T_m$ for all large $r$, because $1/m$ is isolated in $\{0\}\cup\{1/k:k\ge1\}$.
In this case, it is easy to verify that $F x_r\to F x$.

Suppose next that
\[
 x=(0,t,(1-t)\Phi(\omega)),\qquad 0\le t<1.
\]
Then we may write
\[
 x_r=(a_r,t_r,(1-t_r)\Phi(\omega_r)),
\]
where $a_r\to0$, $t_r\to t$, and $\Phi(\omega_r)\to\Phi(\omega)$. If $x_r$ comes from a finite
tower $T_{m_r}$, then the next level has second coordinate $t_r+1/m_r$, which still tends to
$t$ when $m_r\to\infty$. Since $S$ is continuous,
\[
 \Phi(\sigma\omega_r)=S\Phi(\omega_r)\to S\Phi(\omega)=\Phi(\sigma\omega).
\]
Hence $F x_r\to F x$.

Finally suppose $x=q=(0,1,0)$. Then $a_r\to0$, $t_r\to1$, and the fibre weight tends to zero.
The image $Fx_r$ has first coordinate tending to $0$, second coordinate tending to $1$, and
third coordinate bounded by
\[
 (1-t'_r)\sup_{\omega\in\Sigma}\|\Phi(\omega)\|_H
\]
with $t'_r\to1$. Hence $Fx_r\to q=Fq$.
\end{proof}

\subsubsection{Positive $p$-type topological preimage entropy}

\begin{proposition}
The system satisfies
\[
 h_p(F)\ge \frac12\log2>0.
\]
\end{proposition}

\begin{proof}
Fix the common fixed point $q$. For each $m\ge1$, define
\[
 y_\omega=\left(\frac1m,0,\Phi(\omega)\right)\in T_m.
\]
Since the tower has height $m$ and its top is sent to $q$, we have
\[
 F^{m+1}y_\omega=q.
\]
Thus $y_\omega\in F^{-(m+1)}q$.

Let $\ell_m=\lfloor m/2\rfloor$. For each word
$u=(u_0,\ldots,u_{\ell_m-1})\in\{0,1\}^{\ell_m}$, choose
$\omega^u\in\Sigma$ by setting
\[
 \omega^u_i=u_i\quad(0\le i<\ell_m),
 \qquad
 \omega^u_i=0\quad\text{otherwise}.
\]
Let
\[
 E_m=\{y_{\omega^u}:u\in\{0,1\}^{\ell_m}\}\subset F^{-(m+1)}q.
\]
Then $\#E_m=2^{\ell_m}$.
If $u\ne v$, let $r$ be the first coordinate where they differ. Then $r<\ell_m\le m/2$.
At time $r$, the fibre weight is
\[
 1-\frac rm\ge \frac12,
\]
and the shifted symbolic coordinates differ at coordinate $0$. Hence the distance in the fibre is
at least $(1/2)(1/8)=1/16$. Therefore $E_m$ is $(m+1,1/16)$-separated. It follows that
\[
 s(m+1,1/16,F^{-(m+1)}q)\ge 2^{\lfloor m/2\rfloor}.
\]
Thus,
\[
 h_p(F)
 \ge \limsup_{m\to\infty}\frac1{m+1}\log 2^{\lfloor m/2\rfloor}
 =\frac12\log2.
\]
\end{proof}

\subsubsection{Ergodic invariant measures}

\begin{lemma}
Every ergodic invariant probability measure of $(X,F)$ is of one of the following two types:
\begin{enumerate}
\item an ergodic measure supported on a limit layer
\[
 L_t=\{(0,t,(1-t)\Phi(\omega)):\omega\in\Sigma\},\qquad 0\le t<1;
\]
\item the fixed point measure $\delta_q$.
\end{enumerate}
\end{lemma}

\begin{proof}
Note that every point in $T_m$ reaches $p_m$ in at most $m$ steps and then
maps to $q$. Hence no invariant probability measure can give positive measure to a finite tower.
Thus every invariant measure is supported on $L$.
On $L$, the second coordinate $t$ is invariant. Therefore an ergodic measure has a constant
$t$-coordinate. If $t<1$, the layer $L_t$ is conjugate to the two-sided full shift. If $t=1$, the
layer is the single fixed point $q$.
\end{proof}

\subsubsection{All ergodic measures have zero $m$-type empirical preimage entropy}

We now prove the assertion:
\[
 h^*_{m,\nu}(F)=0\qquad\text{for every }\nu\in\Me_F(X).
\]

\begin{lemma}
	Let $\nu$ be an ergodic invariant measure supported on
	\[
	L_{t_0}=\{(0,t_0,(1-t_0)\Phi(\omega)):\omega\in\Sigma\},
	\qquad 0\le t_0<1.
	\]
	Then
	\[
	h^*_{m,\nu}(F)=0.
	\]
\end{lemma}

\begin{proof}
	Fix $\varepsilon>0$. Choose an open interval $J\subset[0,1)$ with
	$t_0\in J$ and $\overline J\subset[0,1)$. Choose a continuous function
	$\psi:[0,1]\to[0,1]$ such that $\psi=1$ on a smaller neighbourhood of
	$t_0$ and $\operatorname{supp}\psi\subset J$. Let $\gamma>0$ be small and set
	\[
	G=\left\{\theta\in\mathcal M(X):
	\int_X \psi(t(z))\,d\theta(z)>1-\gamma
	\right\}.
	\]
	Then $G$ is a neighbourhood of $\nu$, and every orbit segment whose
	empirical measure lies in $G$ spends all but at most
	$\gamma n$ times in $J$.
	
	Since $\overline J\subset[0,1)$, any finite tower orbit segment which reaches a
	tower top, or which reaches the common fixed point $q$, must spend a fixed
	positive proportion of its time outside $J$. Taking $\gamma$ sufficiently
	small, we get that, for all sufficiently large $n$, no point
	$y$ with $\mathcal E_n(y)\in G$ can satisfy
	\[
	F^n y=p_m
	\quad\text{for some }m,
	\]
	or
	\[
	F^n y=q.
	\]
	
	It remains to consider endpoints $x$ which are either non-top points of finite
	towers or points of visible limit layers. If $x$ is a non-top point of a finite
	tower, then $F^{-n}x$ is either empty or consists of one point. 
	If $x$ belongs to a visible limit layer, then $F$ is a
	homeomorphism on that layer, so again $F^{-n}x$ consists of at most one point.
	
	Therefore, for all sufficiently large $n$ and every $x\in X$,
	\[
	s(n,\varepsilon,P_n(x,G))\le 1.
	\]
	It follows that
	\[
	\limsup_{n\to\infty}
	\frac1n\log\sup_{x\in X}s(n,\varepsilon,P_n(x,G))=0.
	\]
	Taking the infimum over all neighbourhoods $G\ni\nu$ and then letting
	$\varepsilon\to0$, we obtain
	\[
	h^*_{m,\nu}(F)=0.
	\]
\end{proof}

\begin{lemma}
	For the fixed point measure $\delta_q$, one has
	\[
	h^*_{m,\delta_q}(F)=0.
	\]
\end{lemma}

\begin{proof}
	Fix $\varepsilon>0$. Recall that the third coordinate is $(1-t)\Phi(\omega)$ and put
	\[
	M_0=\sup_{\omega\in\Sigma}\|\Phi(\omega)\|_H<1.
	\]
	Choose $\rho=\rho(\varepsilon)>0$ so small that $2\rho M_0<\varepsilon/8$. We call a time effectively visible if the second coordinate $t$ satisfies $1-t\geq\rho$, or equivalently $t\leq1-\rho$. If $1-t<\rho$, then the whole symbolic fibre at level $t$ has diameter less than $\varepsilon/8$.
	
	Choose a continuous function $\chi:[0,1]\to[0,1]$ such that $\chi=0$ on $[0,1-\rho]$ and $\chi=1$ near $1$. For $0<\gamma<1$, define
	\[
	G_\gamma=\left\{\theta\in\M(X):\int_X \chi(t(z))\,d\theta(z)>1-\gamma\right\}.
	\]
	Then $G_\gamma$ is a weak-star neighbourhood of $\delta_q$. If $\mathcal E_n(y)\in G_\gamma$, then the orbit segment $y,Fy,\ldots,F^{n-1}y$ has at most $\gamma n$ effectively visible times.
	
	Let $E\subset P_n(x,G_\gamma)$ be an $(n,\varepsilon)$-separated set, where $x\in X$ is arbitrary. We first count possible base orbits, namely the orbits of the first two coordinates. If $x\neq q$, then $x$ fixes the finite tower or the limit layer, and the base orbit is unique. Suppose $x=q$. For $y\in F^{-n}q$, let $\tau(y)$ be the first time at which the orbit of $y$ reaches $q$. Then $0\leq\tau(y)\leq n$, so there are at most $n+1$ possible values of $\tau$.
	
	Fix $\tau\geq1$. If the orbit comes from $T_m$, then its starting level is $m+1-\tau$, so $m\geq\tau-1$. Put $u=1/m$. Before the orbit reaches $q$, the first two coordinates are
	\[
	\left(u,\ 1-(\tau-1-j)u\right),\qquad 0\leq j<\tau.
	\]
	Thus two parameters $u,u'$ give base orbits whose $n$-step base distance is $\tau|u-u'|$. Hence, if these base orbits are $\varepsilon/8$-separated, then $|u-u'|\geq\varepsilon/(8\tau)$. For $\tau\geq2$, all possible $u$ lie in an interval of length at most $1/(\tau-1)$, so the number of mutually $\varepsilon/8$-separated such parameters is at most $1+16/\varepsilon$. The case $\tau=1$ is also bounded by a constant depending only on $\varepsilon$. Hence there is a constant $C_\varepsilon>0$, independent of $n$ and $\tau$, such that all possible base orbits can be covered by at most $C_\varepsilon(n+1)$ balls of radius $\varepsilon/8$ in the base $n$-metric.
	
	Fix one such base ball. Inside this ball, the first two coordinate orbits of any two points are within $\varepsilon/4$ of each other. Along a finite tower, the second coordinate is monotone increasing until the orbit reaches $q$. Therefore the effectively visible times form an initial interval $\{0,1,\ldots,L\}$ with $L+1\leq\gamma n$.
	
	Choose $R=R(\varepsilon)$ so large that agreement of two symbols on $[-R,R]$ implies that their $\Phi$-images are within $\varepsilon/8$ in $H$. If two starting symbolic coordinates agree on $[-R,L+R]$, then at every effectively visible time $0\leq j\leq L$ the shifted symbols agree on $[j-R,j+R]$, so their third coordinates are close. At non-effectively visible times the symbolic fibres already have diameter less than $\varepsilon/8$. Together with the $\varepsilon/4$ control of the first two coordinates, such two points cannot be $(n,\varepsilon)$-separated.
	
	Hence, inside one fixed base ball, an $(n,\varepsilon)$-separated subset has cardinality at most
	\[
	2^{L+2R+1}\leq 2^{\gamma n+2R+1}.
	\]
	Combining this with the base-orbit count gives
	\[
	\#E\leq C_\varepsilon(n+1)2^{\gamma n+2R+1}.
	\]
	Thus
	\[
	\limsup_{n\to\infty}\frac1n\log\sup_{x\in X}s(n,\varepsilon,P_n(x,G_\gamma))
	\leq \gamma\log2.
	\]
	Since $\gamma>0$ is arbitrary, we get $h_m^*(F,\delta_q,\varepsilon)=0$. Letting $\varepsilon\to0$ gives $h^*_{m,\delta_q}(F)=0$.
\end{proof}

Combining the preceding two lemmas, we obtain
\[
 \sup_{\nu\in\Me_F(X)}h^*_{m,\nu}(F)=0.
\]

\section{Relations among empirical and pointwise metric preimage entropies}\label{sec:relations}

\subsection{Basic properties}
In this subsection, we give the power inequality and product inequality for the empirical metric preimage entropy.
\begin{proposition}[Power inequality]\label{pro:power inequality*}
Let $f:X\to X$ be continuous, $\mu\in\M_f(X)$ and let $\ell\geq1$. Then
\[
 h^{*}_{m,\mu}(f^\ell)
 \leq \ell\,h^{*}_{m,\mu}(f),\quad
  h^{*}_{p,\mu}(f^\ell)
 \leq \ell\,h^{*}_{p,\mu}(f).
\]
\end{proposition}

\begin{proof}
Write $g=f^\ell$.  Fix $\epsilon>0$ and let $G$ be a neighbourhood of $\mu$ in $\M(X)$.  Consider the continuous affine map $A_\ell:\M(X)\to\M(X)$ defined by
\[
A_\ell(\theta)=\frac1\ell\sum_{r=0}^{\ell-1}f^r_*\theta.
\]
Since $A_\ell(\mu)=\mu$, there exists a weak-star neighbourhood $H$ of $\mu$ such that
\[
 A_\ell(H)\subset G.
\]

For every $y\in X$ and $n\geq1$ we have
\[
 \mathcal{E}_{\ell n}^f(y)
 =\frac1{\ell n}\sum_{j=0}^{\ell n-1}\delta_{f^j y}
 =\frac1\ell\sum_{r=0}^{\ell-1}f^r_*
 \left(\frac1n\sum_{i=0}^{n-1}\delta_{f^{\ell i}y}\right)
 =A_\ell(\mathcal{E}_n^g(y)).
\]
Therefore, if $\mathcal{E}_n^g(y)\in H$ then $\mathcal{E}_{\ell n}^f(y)\in G$.

Moreover,
\[
 g^{-n}x=(f^\ell)^{-n}x=f^{-\ell n}x.
\]
Also, if two points are $(n,\epsilon)$-separated for $g$, then they are $(\ell n,\epsilon)$-separated for $f$. Hence, for every $x\in X$,
\[
 s_g(n,\epsilon,P_n^g(x,H))
 \leq s_f(\ell n,\epsilon,P_{\ell n}^f(x,G)),
\]
which gives
\[
 h^{*}_{m,\mu}(g,\epsilon;H)
 \leq \ell\, h^{*}_{m,\mu}(f,\epsilon;G).
\]
Consequently,
\[
 h^{*}_{m,\mu}(f^\ell)
 \leq \ell\, h^{*}_{m,\mu}(f).
\]

The proof for $h^{*}_{p,\mu}(f)$ is similar.
\end{proof}

Let $(X, d_X)$ and $(Y, d_Y)$ be compact metric spaces, and let $f: X \to X$ and $g: Y \to Y$ be continuous maps. Equip $X \times Y$ with the max metric
\[
d\big((x,y), (x',y')\big) = \max\{d_X(x,x'),\, d_Y(y,y')\}.
\]

\begin{proposition}[Product inequality]
Let $(X, d_X)$ and $(Y, d_Y)$ be compact metric spaces, and let $f: X \to X$ and $g: Y \to Y$ be continuous maps. For any $\mu \in \mathcal{M}_f(X)$ and $\nu \in \mathcal{M}_g(Y)$, we have
\[
 h^{*}_{m,\mu\times\nu}(f\times g)
 \leq h^{*}_{m,\mu}(f)+h^{*}_{m,\nu}(g),\quad
  h^{*}_{p,\mu\times\nu}(f\times g)
 \leq h^{*}_{p,\mu}(f)+h^{*}_{p,\nu}(g).
\]
\end{proposition}

\begin{proof}
Fix $\epsilon>0$.  Let $G_X$ be a neighbourhood of $\mu$ in $\M(X)$ and let $G_Y$ be a neighbourhood of $\nu$ in $\M(Y)$.  Let $\pi_X:X\times Y\to X$ and $\pi_Y:X\times Y\to Y$ be the projection mappings.  Define
\[
 G=\{\lambda\in\M(X\times Y): (\pi_X)_*\lambda\in G_X,\; (\pi_Y)_*\lambda\in G_Y\}.
\]
Then $G$ is a neighbourhood of $\mu\times\nu$.

If $((a,b))\in P_n^{f\times g}((x,y),G)$, then
\[
 a\in P_n^f(x,G_X),\qquad b\in P_n^g(y,G_Y).
\]
Hence
\[
 P_n^{f\times g}((x,y),G)\subset P_n^f(x,G_X)\times P_n^g(y,G_Y).
\]

For subsets $A\subset X$ and $B\subset Y$, we have the standard product estimate
\[
 s_{f\times g}(n,\epsilon,A\times B)\le s_f(n,\epsilon/4,A)\,s_g(n,\epsilon/4,B).
\]
Consequently,
\[
 \sup_{(x,y)}s_{f\times g}(n,\epsilon,P_n^{f\times g}((x,y),G))
 \le\Bigl(\sup_x s_f(n,\epsilon/4,P_n^f(x,G_X))\Bigr)
      \Bigl(\sup_y s_g(n,\epsilon/4,P_n^g(y,G_Y))\Bigr).
\]
Taking logarithms, dividing by $n$, and taking limsup gives
\[
 h^{*}_{m,\mu\times\nu}(f\times g,\epsilon;G)
 \le h^{*}_{m,\mu}(f,\epsilon/4;G_X)+h^{*}_{m,\nu}(g,\epsilon/4;G_Y).
\]
This yields
\[
 h^{*}_{m,\mu\times\nu}(f\times g,\epsilon)
 \le h^{*}_{m,\mu}(f,\epsilon/4)+h^{*}_{m,\nu}(g,\epsilon/4).
\]
Letting $\epsilon\to0$ proves the desired inequality.

A similar argument applies to $h^{*}_{p,\mu}(f)$.
\end{proof}

\subsection{Proof of Theorem \ref{thm:relations for ergodic case}}

\begin{lemma}\cite[Theorem A]{WuZhu2022}\label{lem:SMB theorem}
Let $f:X\to X$ be continuous and $\alpha$ be a finite Borel partition of $X$. If $\mu$ is ergodic, then
\[
\lim_{n\to \infty}\frac1n I_\mu(\alpha_0^{n-1}\mid f^{-n}\mathcal{B})(y)= h_{m,\mu}(f,\alpha)
\]
for $\mu$-almost every $y$.
\end{lemma}

\begin{lemma}\cite[Proposition 2.8]{ChengLiWu2023}\label{lem:relation-lower bound}
If $f:X\to X$ is a continuous map with uniform separation of preimages and $\mu\in\M_f(X)$, then
\[
 h_{m,\mu}^{*}(f)\le h_{m,\mu}(f).
\]
\end{lemma}

By Lemma \ref{lem:relation-lower bound}, to prove Theorem \ref{thm:relations for ergodic case}, it suffices to establish the following proposition.

\begin{proposition}
Let $f:X\to X$ be continuous. If $\mu\in\M_f^e(X)$, then
\[
 h_{p,\mu}^{*}(f)\ge h_{m,\mu}(f).
\]
\end{proposition}

\begin{proof}
Let $\alpha=\{A_1,\ldots,A_s\}$ be a finite partition and let $h=h_{m,\mu}(f,\alpha)$. Fix an arbitrary neighbourhood $G\ni\mu$, $\eta>0$ and $\gamma>0$. For each atom $A_i$ choose a compact set
$K_i\subset A_i$ such that, with
\[
 K=\bigcup_{i=1}^s K_i,
 \qquad
 B_0=X\setminus K,
\]
we have
\[
 \mu(B_0)<\gamma/4.
\]
The nonempty compact sets $K_i$ are pairwise disjoint. Therefore
\[
 \delta=\min_{i\ne j}{\rm dist}(K_i,K_j)>0.
\]

Define $Z_n(G,\eta,\gamma)$ to be the set of all $y\in X$ satisfying:
\begin{enumerate}
\item $\mathcal{E}_n(y)\in G$;
\item $I_\mu(\alpha_0^{n-1}\mid f^{-n}\mathcal{B})(y)\ge n(h-\eta)$;
\item $\#\{0\le j<n:f^j y\in B_0\}\le \gamma n$.
\end{enumerate}
By the Birkhoff ergodic theorem and Lemma \ref{lem:SMB theorem}, we get
\[
 \lim_{n\to\infty}\mu(Z_n(G,\eta,\gamma))=1.
\]

For each $n\ge1$, let
\[
\mu=\int_X \mu_x^{f^{-n}\B}\,d\mu(x)
\]
be a measure decomposition over $f^{-n}\B$.  More precisely, for $\mu$-a.e. $x$,
\[
\mu_x^{f^{-n}\B}\bigl(f^{-n}(f^n x)\bigr)=1,
\]

Define
\[
\eta_{n,f^n x}=\mu_x^{f^{-n}\B}
\]
for $\mu$-a.e. $x$.  This is well defined modulo null sets based on the properties of the conditional measures. Therefore,
there exists a family of probability measures $\{\eta_{n,z}:z\in X\}$
such that
\[
\eta_{n,z}(f^{-n}z)=1
\qquad\text{for }\mu\text{-a.e. }z,
\]
and
\[
\mu(C)=\int_X \eta_{n,z}(C\cap f^{-n}z)\,d\mu(z)
\]
for every Borel set $C\subset X$ since $\mu$ is $f$-invariant.
Moreover, if \(C\in\alpha_0^{n-1}\) and \(y\in C\cap f^{-n}z\), then the conditional information can be written as
\[
 I_\mu(\alpha_0^{n-1}\mid f^{-n}\B)(y)
 = -\log \eta_{n,z}(C).
\]

Define
\[
B_n(G,\eta,\gamma)
=\left\{z\in X:
\eta_{n,z}\bigl(Z_n(G,\eta,\gamma)\cap f^{-n}z\bigr)>\frac12
\right\}.
\]
Because $\mu(Z_n(G,\eta,\gamma))\to1$, it follows that
\[
\mu(B_n(G,\eta,\gamma))\to1.
\]
Indeed, if $\mu(B_n(G,\eta,\gamma))\le1-\delta$ for some $\delta>0$ and infinity many $n$, then
\[
\begin{aligned}
 \mu(Z_n(G,\eta,\gamma))
 &=\int_X \eta_{n,z}(Z_n(G,\eta,\gamma)\cap f^{-n}z)\,d\mu(z)\\
 &\le \mu(B_n(G,\eta,\gamma))\cdot 1+(1-\mu(B_n(G,\eta,\gamma)))\cdot\frac12\\
 &=\frac12+\frac12\mu(B_n(G,\eta,\gamma))\\
 &\le1-\frac\delta2,
\end{aligned}
\]
contradicting $\mu(Z_n(G,\eta,\gamma))\to1$.

Let $\{G_r\}_{r\ge1}$ be a countable neighbourhood basis at $\mu$.  For each fixed $r$, choose an increasing sequence $n_j^{(r)}\to\infty$ such that
\[
\mu\left(B_{n_j^{(r)}}(G_r,\eta,\gamma)^c\right)<2^{-j}.
\]
By the Borel-Cantelli lemma,
\[
\mu\left(\liminf_{j\to\infty}B_{n_j^{(r)}}(G_r,\eta,\gamma)\right)=1.
\]
Taking the intersection over all $r$ still gives a full measure set.  Choose
\[
        x_*\in
        \bigcap_{r=1}^{\infty}
        \liminf_{j\to\infty}B_{n_j^{(r)}}(G_r,\eta,\gamma).
\]
Then for every $r$, there exists $J_r$ such that for all $j\ge J_r$,
\[
        x_*\in B_{n_j^{(r)}}(G_r,\eta,\gamma).
\]
Equivalently,
\[
\eta_{n_j^{(r)},x_*}
\left(
Z_{n_j^{(r)}}(G_r,\eta,\gamma)
\cap f^{-n_j^{(r)}}x_*
\right)>\frac12
\]
for all sufficiently large $j$.  This fixed endpoint $x_*$ is the key point. 

Fix $r$ and write $n_j=n_j^{(r)}$.  For all large $j$, put
\[
W_{n_j}=Z_{n_j}(G_r,\eta,\gamma)
\cap f^{-n_j}x_*.
\]
Then
\[
        \eta_{n_j,x_*}(W_{n_j})>\frac12.
\]
Define
\[
        \mathcal A_{n_j}
        =\{C\in\alpha_0^{n_j-1}:C\cap W_{n_j}\ne\emptyset\}.
\]
Take $C\in\mathcal A_{n_j}$ and choose $y\in C\cap W_{n_j}$.
Since $y\in Z_{n_j}(G_r,\eta,\gamma)$, one has
\[
I_\mu(\alpha_0^{n_j-1}\mid f^{-n_j}\B)(y)=-\log \eta_{n_j,x_*}(C)
        \ge n_j(h-\eta).
\]
So
\[
        \eta_{n_j,x_*}(C)\le e^{-n_j(h-\eta)},
\]
which together with $\eta_{n_j,x_*}(W_{n_j})>1/2$ implies
\[
        \#\mathcal A_{n_j}\ge \frac12 e^{n_j(h-\eta)}.
\]
Choose one point from each atom $C\in\mathcal A_{n_j}$.  This gives a set $R_{n_j}\subset W_{n_j}$ such that
\[
        \#R_{n_j}\ge \frac12 e^{n_j(h-\eta)},
\]
and distinct points in $R_{n_j}$ have distinct $(\alpha,n_j)$-names, where the $(\alpha,n_j)$-name of $y$ is
\[
 \omega(y)=(\omega_0(y),\ldots,\omega_{n_j-1}(y))\in\{1,\ldots,s\}^{n_j}
\]
such that $f^i y\in A_{\omega_i(y)}$.

For $\omega=(\omega_0,\ldots,\omega_{n-1}), \omega'=(\omega'_0,\ldots,\omega'_{n-1})\in\{1,\ldots,s\}^n$,
the Hamming distance $d_H(\omega,\omega')$ on $\{1,\ldots,s\}^n$ is defined by
\[
\frac{1}{n}\#\{i\in\{0,1,\ldots,n-1\}: \omega_i\neq \omega'_i\}.
\]
Put
\[
 V_s(n,r)=\sum_{j=0}^{r}\binom{n}{j}(s-1)^j.
\]
For $0<\gamma<1/4$, the standard Stirling's approximation gives
\[
 V_s(n,\lfloor 2\gamma n\rfloor)
 \le \exp\bigl(n\Theta_s(\gamma)+o(n)\bigr),
\]
where
\[
 \Theta_s(\gamma)=H(2\gamma)+2\gamma\log(s-1),\quad
 H(t)=-t\log t-(1-t)\log(1-t),
\]
and
\[
 \Theta_s(\gamma)\to0\quad\text{as }\gamma\to0.
\]

Apply the greedy selection algorithm to the set of names $\omega(R_{n_j})$: choose one name, delete all names within Hamming distance $2\gamma$, and continue. This produces a subset
\[
 Q_{n_j}\subset R_{n_j}
\]
such that for any distinct $y,z\in Q_{n_j}$,
\[
 d_H(\omega(y),\omega(z))>2\gamma,
\]
and
\[
 \#Q_{n_j}\ge
 \frac{\#R_{n_j}}{V_s({n_j},\lfloor2\gamma {n_j}\rfloor)}.
\]
Consequently,
\[
 \limsup_{j\to\infty}\frac1n_j\log\#Q_{n_j}
 \ge h-\eta-\Theta_s(\gamma).
\]

We claim that $Q_{n_j}$ is $(n_j,\delta)$-separated. Let $y,z\in Q_{n_j}$ be distinct. Their names differ in more than $2\gamma {n_j}$ positions. Since both $y$ and $z$ belong to $Z_{n_j}(G_r,\eta,\gamma)$, each of them visits $B_0$ at most $\gamma n_j$ times. The union of their bad times has cardinality at most $2\gamma n_j$. Therefore there exists some $0\le i<n_j$ such that
\[
 \omega_i(y)\ne\omega_i(z),
 \qquad
 f^i y\notin B_0,
 \qquad
 f^i z\notin B_0.
\]
At this time,
\[
 f^i y\in K_{\omega_i(y)},
 \qquad
 f^i z\in K_{\omega_i(z)},
\]
and these two compact sets are different. Hence
\[
 d(f^i y,f^i z)\ge\delta.
\]
This yields $d_{n_j}(y,z)\ge\delta$.
So $Q_{n_j}$ is $(n_j,\delta)$-separated.

Since
\[
        Q_{n_j}\subset W_{n_j}\subset P_{n_j}(x_*,G_r),
\]
we have
\[
        s(n_j,\delta,P_{n_j}(x_*,G_r))\ge \#Q_{n_j}.
\]
Consequently,
\[
        h^*_{p,\mu}(f,x_*,\delta;G_r)
        \ge h-\eta-\Theta_s(\gamma).
\]

Let $G\ni\mu$ be any neighbourhood.  Since $\{G_r\}$ is a neighbourhood basis, choose $r$ with $G_r\subset G$.  Then
\[
        P_n(x_*,G_r)\subset P_n(x_*,G),
\]
so
\[
        h^*_{p,\mu}(f,x_*,\delta;G)
        \ge h^*_{p,\mu}(f,x_*,\delta;G_r)
        \ge h-\eta-\Theta_s(\gamma).
\]
Taking the infimum over all $G\ni\mu$, we get
\[
h^*_{p,\mu}(f)\ge h^*_{p,\mu}(f,x_*,\delta)
        \ge h-\eta-\Theta_s(\gamma).
\]
Letting first $\eta\to0$ and then $\gamma\to0$, we obtain
\[
        h^*_{p,\mu}(f)
        \ge h_{m,\mu}(f,\alpha).
\]
Since $\alpha$ is arbitrary,
\[
 h_{p,\mu}^{*}(f)
 \ge h_{m,\mu}(f).
\]
\end{proof}

\subsection{Proof of Theorem \ref{thm:relations for general case}}

Let
\[
\Sigma_k^+ = \{0,1,\ldots,k-1\}^{\mathbb{N}_0}
\]
be the one-sided full shift space, where $\mathbb{N}_0=\{0,1,2,\ldots\}$.
Let
\[
\sigma:\Sigma_k^+\to\Sigma_k^+,\qquad
(\sigma x)_j=x_{j+1}
\]
be the left shift. This map is continuous and non-invertible. Each point has exactly $k$ preimages.

Let $p=(p_1,\ldots,p_k)$ be a probability vector with $p_i\geq 0$ and $\sum_{i=1}^k p_i=1$. Let $\mu=p^{\mathbb{N}_0}$ be the Bernoulli product measure. Then $\mu$ is $\sigma$-invariant. For the Bernoulli measure, it is well-known that
\[
h_\mu(\sigma)=H(p):=-\sum_{i=1}^k p_i\log p_i,
\]
with the convention $0\log 0=0$.

\begin{proposition}\label{pro:equality for positively expansive}
Let $f$ be forward expansive and $\mu\in\M_f(X)$. Then
\[
h_\mu(f)=F_\mu(f)=h_{m,\mu}(f).
\]
\end{proposition}

\begin{proof}
Choose a finite partition $\alpha$ with $\operatorname{diam}\alpha<c$.
By forward expansiveness, we have
\[
\mathcal{B}=\bigvee_{j=0}^{\infty}f^{-j}\alpha.
\]
Taking the inverse image under $f$, we get
\[
f^{-1}\mathcal{B}=\bigvee_{j=1}^{\infty}f^{-j}\alpha.
\]
Hence
\[
h_\mu(f)=h_\mu(f,\alpha)=H_\mu\left(\alpha\mid \bigvee_{j=1}^{\infty}f^{-j}\alpha\right)
=H_\mu(\alpha\mid f^{-1}\mathcal{B}).
\]

Moreover, by $\alpha\vee f^{-1}\mathcal{B}=\mathcal{B}$ we get
\[
F_\mu(f)
=H_\mu(\mathcal{B}\mid f^{-1}\mathcal{B})
=H_\mu(\alpha\vee f^{-1}\mathcal{B}\mid f^{-1}\mathcal{B})
=H_\mu(\alpha\mid f^{-1}\mathcal{B}).
\]
Combining the two equalities yields
\[
h_\mu(f)=F_\mu(f).
\]
Since $h_{m,\mu}(f)=F_\mu(f)$ (see \cite[Proposition 3.3]{WuZhu2022}), we get the desired equality.
\end{proof}

\begin{example}\label{exa:relation1}
Let
\[
X=\Sigma_2^+\sqcup \Sigma_3^+
\]
and
$f:X\to X$ be the disjoint union of the one-sided shifts
$\sigma_2:\Sigma_2^+\to \Sigma_2^+$ and $\sigma_3:\Sigma_3^+\to \Sigma_3^+$. Let $\nu_k$ be the uniform
Bernoulli measure on $\Sigma_k^+$ for $k=2,3$, and set
\[
\mu=\frac12\nu_2+\frac12\nu_3.
\]
Then
\[
h^{*}_{m,\mu}(f)=0
\quad\text{whereas}\quad
h_{m,\mu}(f)=\frac12\log 2+\frac12\log 3>0.
\]
\end{example}

\begin{proof}
By Proposition \ref{pro:equality for positively expansive}, we have $ h_{m,\nu_k}(\sigma_k) = \log k $. Then, by the affinity of $ h_{m,\mu}(f) $ (see \cite[Proposition 2.12]{WuZhu2021}),
\[
h_{m,\mu}(f)=\frac12 h_{m,\nu_2}(\sigma_2)+\frac12 h_{m,\nu_3}(\sigma_3)
=\frac12\log2+\frac12\log3.
\]

Let $\varphi:X\to\mathbb R$ be the continuous function
\[
\varphi|_{\Sigma_2^+}=1,\qquad
\varphi|_{\Sigma_3^+}=0.
\]
Then $\int_X\varphi\,d\mu=\frac12$. Choose the neighbourhood of $\mu$:
\[
G=\left\{\theta\in\mathcal M(X):
\left|\int\varphi\,d\theta-\frac12\right|<\frac14\right\}.
\]
Every orbit segment of $f$ is entirely contained in either $\Sigma_2^+$ or $\Sigma_3^+$. Hence,
for every $y\in X$ and every $n$,
\[
\int\varphi\,d\mathcal{E}_n(y)
=1 \quad\text{or}\quad 0.
\]
Therefore $\mathcal{E}_n(y)\notin G$ for all $y,n$. Consequently
\[
P_n(x,G)=\emptyset
\qquad\text{for every }x\in X\text{ and every }n\geq 1,
\]
which immediately gives
\[
h^{*}_{m,\mu}(f)=0.
\]
\end{proof}

The following example is taken from Wang \cite[Lemma 2.1]{Wang2024}.

\begin{example}\label{exa:relation2}
Let $A = \{0, 1, 2\}$ and endow $A^{\mathbb{N} \times \mathbb{N}}$ with the product topology induced by the discrete topology on $A$.
Denote by $f: A^{\mathbb{N} \times \mathbb{N}} \to A^{\mathbb{N} \times \mathbb{N}}$ the left shift map acting on rows, i.e.,
\[
(f x)_{m,i} = x_{m,i+1} \quad \text{for } m, i \in \mathbb{N}.
\]
For each array $x = (x_{m,i})_{m,i \ge 0}$, let $i_0(x)$ be the smallest index $i \ge 0$ such that $x_{0,i} = 0$; if no such index exists, set $i_0(x) = \infty$.

Define $X \subset A^{\mathbb{N} \times \mathbb{N}}$ as the set of arrays satisfying:
\begin{enumerate}
  \item For all $i \ge i_0(x)$ and all $m \ge 0$, we have $x_{m,i} = 0$;
  \item For all $0 \le i < i_0(x)$ and all $m \ge 0$, we have $x_{m,i} \in \{1, 2\}$, and if $m \ge 1$ and $i \ge 1$, then $x_{m,i} = x_{m-1,i-1}$.
\end{enumerate}
Then we have $h_p(f) \ge \log 2$ and $h_{m,\mu}(f) = 0$ for every $\mu \in \mathcal{M}(X, f)$.
By Theorem \ref{thm:variational principle},
\[
h_p(f) = \sup_{\mu \in \mathcal{M}(X, f)} h_{p,\mu}^*(f).
\]
Thus, there exists some $\mu \in \mathcal{M}(X, f)$ such that $h_{p,\mu}^*(f) > 0$. Consequently, $h_{p,\mu}^*(f) > h_{m,\mu}(f)$.
\end{example}

Theorem \ref{thm:relations for general case} follows directly from Example \ref{exa:relation1} and Example \ref{exa:relation2}.

\section{The resolving-partition property}\label{sec:resolving-partition property}

\subsection{Proof of Theorem \ref{thm:resolving partition}}

\begin{definition}
Let $f:X\to X$ be a continuous map. A finite Borel partition $\alpha$ of $X$ is called \emph{$f$-resolving} if
\[
  \alpha\vee f^{-1}\B=\B.
\]
Equivalently, $f$ is one-to-one on each atom of $\alpha$.
\end{definition}

\begin{definition}
Let $f:X\to X$ be a continuous map. We say that $f$ has the \emph{resolving-partition property}, if for every $\delta>0$ and every $\mu\in\M_f(X)$ there exists a finite Borel partition $\alpha$ such that
\begin{enumerate}
\item ${\rm diam}\alpha<\delta$;
\item $\mu(\partial\alpha)=0$;
\item $\alpha$ is $f$-resolving.
\end{enumerate}
Here $\partial\alpha=\bigcup_{A\in\alpha}\partial A$.
\end{definition}

Uniform separation of preimages implies resolving-partition property: take any finite partition of diameter smaller than the separation constant and with null boundary.  The converse is false.  The full tent map has resolving-partition property but not uniform separation of preimages; see Corollary \ref{cor:tent map}.

\begin{lemma}\label{lem:iterated-resolving}
If $\alpha$ is $f$-resolving, then for every $n\ge1$,
\[
  \alpha_0^{n-1}\vee f^{-n}\B=\B.
\]
\end{lemma}

\begin{proof}
It is enough to prove that two points in the same atom of $\alpha_0^{n-1}\vee f^{-n}\B$ must be equal.  Suppose $y$ and $x$ lie in the same atom.  Then
\[
  f^i y\in \alpha(f^i x),\qquad 0\le i\le n-1,
\]
and $f^n y=f^n x$.  Since $f^{n-1}y$ and $f^{n-1}x$ lie in the same atom of $\alpha$, and since $\alpha$ is resolving, $f(f^{n-1}y)=f(f^{n-1}x)$ implies
\[
 f^{n-1}y=f^{n-1}x.
\]
Repeating the same argument gives $f^{n-2}y=f^{n-2}x$, then $\ldots$, and finally $y=x$. This ends the proof.
\end{proof}

Before proving the lemma, we record a measurability observation which will be used below. Fix $n\in\mathbb N$ and put $T=f^n$. Let $d_n$ be the Bowen metric of length $n$. If $\xi$ is a finite Borel partition of $X$ and $\{\mu_{n,x}\}_{x\in X}$ is a disintegration of $\mu$ over $f^{-n}\mathcal B$, then the function
\[
x\mapsto H_{\mu_{n,x}}(\xi)
\]
is $\mu$-measurable. Indeed, for every atom $C$ of $\xi$, the function $x\mapsto \mu_{n,x}(C)$ is measurable by the definition of disintegration, and
\[
H_{\mu_{n,x}}(\xi)
=
-\sum_{C\in\xi}\mu_{n,x}(C)\log \mu_{n,x}(C)
\]
is a finite sum of measurable functions.

We shall also use the measurability of the fibrewise separated-cardinality functions. For fixed $n$ and $\varepsilon>0$, define
\[
N_{n,\varepsilon}(x)=s(n,\varepsilon,f^{-n}x),
\]
with the harmless convention $N_{n,\varepsilon}(x)=1$ if $f^{-n}x=\emptyset$. For each $k\geq1$, the set $\{x:N_{n,\varepsilon}(x)\geq k\}$ is analytic, since it is the projection of the Borel set of all $(x,y_1,\ldots,y_k)\in X^{k+1}$ such that $f^n y_i=x$ and $d_n(y_i,y_j)>\varepsilon$ for all $i\neq j$. Hence $N_{n,\varepsilon}$, and therefore $x\mapsto \log N_{n,\varepsilon}(x)$, is universally measurable. In particular it is measurable with respect to the completion of every Borel probability measure. Since $\mu$ is $f$-invariant, we may integrate this function and use
\[
\int_X \log s(n,\varepsilon,f^{-n}(f^n x))\,d\mu(x)
=
\int_X \log s(n,\varepsilon,f^{-n}x)\,d\mu(x).
\]

Finally, for the finite partition $\beta$ used below, the function
\[
x\mapsto \#\mathcal H_{n,x},
\qquad
\mathcal H_{n,x}
=
\{B\in\beta_0^{n-1}: B\cap f^{-n}(f^n x)\neq\emptyset\},
\]
is also universally measurable. Indeed, for each atom $B\in\beta_0^{n-1}$, the set
\[
\{x:B\cap f^{-n}(f^n x)\neq\emptyset\}
=
\{x:\exists y\in B \text{ such that } f^n y=f^n x\}
\]
is analytic, and $\#\mathcal H_{n,x}$ is a finite sum of the corresponding indicator functions. Thus all integrals appearing in the proof are well defined.

The following result was proved by \cite[Theorem 5.3]{ShiYanZeng2023}. Here we present a relatively simple and direct proof.

\begin{lemma}\label{thm:leq hp}
Let $f:X\to X$ be a continuous map. For any $\mu\in\M_f(X)$ we have
$$
h_{m,\mu}(f)\leq h_p(f).
$$
\end{lemma}

\begin{proof}
Let $\mu=\int_X\mu_{n,x}\,d\mu(x)$ be the disintegration of $\mu$ over $f^{-n}\mathcal{B}$ for $n\in\mathbb{N}$. So we have $\mu_{n,x}(f^{-n}(f^nx))=1$ for $\mu$-almost every $x\in X$.

Let $\alpha=\{A_1,\ldots,A_s\}$ be a finite partition. Then one can choose compact sets $B_i\subset A_i$ such that $H_\mu(\alpha|\beta)\leq1$, where $\beta=\{B_0,B_1,\ldots,B_s\}$ and $B_0=X\setminus\bigcup_{i=1}^s B_i$. By a standard approach we obtain
$$
h_{m,\mu}(f,\alpha)\leq h_{m,\mu}(f,\beta)+1.
$$

For $n\in\mathbb{N}$ and $x\in X$, let
$$
\mathcal{H}_{n,x}=\{B\in\beta_0^{n-1}:B\cap f^{-n}(f^nx)\neq\emptyset\}.
$$
Define $\mathcal{U}=\{B_0\cup B_1,\ldots,B_0\cup B_s\}$. Then $\mathcal{U}$ is a finite open cover of $X$. Let $2\varepsilon$ be a Lebesgue number of $\mathcal{U}$ and $E_{n,x}\subset f^{-n}(f^nx)$ be an $(n,\varepsilon)$-separated set with
$$
\#E_{n,x}=s(n,\varepsilon,f^{-n}(f^nx)).
$$
So $E_{n,x}$ is also an $(n,\varepsilon)$-spanning set for $f^{-n}(f^nx)$.
Given $B\in\mathcal{H}_{n,x}$, we pick $y\in B\cap f^{-n}(f^nx)$
and $z(B)\in E_{n,x}$ such that $d_n(y,z(B))<\varepsilon$.

We claim that $\#\{B\in\mathcal{H}_{n,x}:z(B)=z\}\leq 2^n$ for every $z\in E_{n,x}$.
To this end, suppose $z(B)=z(B')$. Then there exist $y\in B\cap f^{-n}(f^nx)$ and $y'\in B'\cap f^{-n}(f^nx)$ such that $d_n(y,z(B))<\varepsilon$ and $d_n(y',z(B'))<\varepsilon$. Hence, we have $d_n(y,y')<2\varepsilon$, which shows that $f^iy$ and $f^iy'$ are in the same element of $\mathcal{U}$, $0\leq i\leq n-1$. So the claim is true.

It follows from the claim that
$$
\#\mathcal{H}_{n,x}\leq 2^n\cdot s(n,\varepsilon,f^{-n}(f^nx)).
$$
Therefore,
\begin{align*}
h_{m,\mu}(f,\beta)&=\limsup_{n\to\infty}\frac{1}{n}H_\mu(\beta_0^{n-1}|f^{-n}\mathcal{B})\\
&=\limsup_{n\to\infty}\frac{1}{n}\int_X H_{\mu_{n,x}}(\beta_0^{n-1})\,d\mu(x)\\
&\leq\limsup_{n\to\infty}\frac{1}{n}\int_X\log\#\mathcal{H}_{n,x}\,d\mu(x)\\
&\leq\log 2+\limsup_{n\to\infty}\frac{1}{n}\int_X\log s(n,\varepsilon,f^{-n}(f^nx))\,d\mu(x)\\
&=\log 2+\limsup_{n\to\infty}\frac{1}{n}\int_X\log s(n,\varepsilon,f^{-n}(x))\,d\mu(x)\\
&\leq\log 2+\sup_{x\in X}\limsup_{n\to\infty}\frac{1}{n}\log s(n,\varepsilon,f^{-n}(x))\\
&\leq\log 2+h_p(f).
\end{align*}
Furthermore, it holds that
$$
h_{m,\mu}(f)\leq h_p(f)+\log2+1.
$$
Now apply this result to $f^k$, we have
$$
k h_{m,\mu}(f)\leq kh_p(f)+\log2+1,
$$
where \cite[Proposition 2.6]{WuZhu2022} and \cite[Theorem 5.1]{NiteckiPrzytycki1999} are used. So
$$
h_{m,\mu}(f)\leq h_p(f).
$$
\end{proof}

\begin{theorem}
If the continuous map $f:X\to X$ has resolving-partition property, then
\[
h_p(f)=h_m(f)=\sup_{\mu\in\M_f(X)}h_{m,\mu}(f)
       =\sup_{\mu\in\M_f^e(X)}h_{m,\mu}(f).
\]
\end{theorem}

\begin{proof}
The inequality $\sup_\mu h_{m,\mu}(f)\le h_p(f)$ follows from Lemma~\ref{thm:leq hp}. The reverse inequality can be obtained by an approach similar to the proof of \cite[Theorem~B]{WuZhu2021}, using the key Lemma~\ref{lem:iterated-resolving}. The variational principle over ergodic invariant measures then follows from the ergodic decomposition property of $h_{m,\mu}(f)$ (see \cite[Theorem~2.13]{WuZhu2022}).
\end{proof}

\subsection{Interval maps and resolving-partition property}

Let $C(f)$ be the set of genuine turning points of a continuous interval map $f$.

\begin{proposition}\label{pro:interval maps}
Let $f:I\to I$ be continuous. Suppose $C(f)$ is finite and $f$ is strictly monotone on each component of $I\setminus C(f)$. Then the following are equivalent:
\begin{enumerate}
  \item[(i)] no point of $C(f)$ is periodic;
  \item[(ii)] $f$ has the resolving-partition property.
\end{enumerate}
\end{proposition}

\begin{proof}
First, assume that no turning point is periodic. Let $\mu\in\M_f(I)$. If $\mu(\{x\})>0$, then $x$ must be periodic. Otherwise, the forward orbit of $x$ would contain infinitely many distinct points, each with measure at least $\mu(\{x\})$, which is impossible because $\mu$ is finite. Therefore every non-periodic point has zero mass for every invariant measure. Since the finite set $C(f)$ contains no periodic point, we have $\mu(C(f))=0$.

Given $\delta>0$, choose a finite set $P\subset I$ that contains $C(f)$ and some additional cut points, so that each component of $I\setminus P$ has length less than $\delta$. Since a probability measure has at most countably many atoms, we can choose the additional cut points outside the atoms of $\mu$. Thus $\mu(P)=0$. The interval partition determined by $P$ has diameter less than $\delta$, its boundary is contained in $P$, and the boundary has $\mu$-measure zero. On each atom, the map is strictly monotone, hence one-to-one. Therefore the partition is $f$-resolving and $f$ has resolving-partition property.

Conversely, suppose $c\in C(f)$ is periodic and let $\mu$ be the invariant probability measure equidistributed on the orbit of $c$. Because a turning point is locally non-injective, every $f$-resolving partition must contain $c$ in its boundary. Hence $\mu(\partial\alpha)>0$ for every $f$-resolving partition $\alpha$, which contradicts the resolving-partition property.
\end{proof}

\begin{corollary}\label{cor:tent map}
The full tent map
\[
 T(x)=1-|1-2x|
\]
has resolving-partition property. Its only genuine turning point is $c=1/2$, and
\[
 T(c)=1,\qquad T(1)=0,\qquad T(0)=0.
\]
Thus $c$ is not periodic. Proposition~\ref{pro:interval maps} applies.
\end{corollary}

\section{The variational principle for preimage pressure}\label{sec:variational principle for preimage pressure}

In this section, we are going to prove the variational principles for preimage pressures $P_m(f,\phi)$ and $P_p(f,\phi)$, which were introduced and investigated by Li, Wu and Zhu \cite{LiWuZhu2020}.
We first recall the definition of topological preimage pressure.
Let $(X,d)$ be a compact metric space and let $f:X\to X$ be continuous.  For a continuous potential $\phi\in C(X)$ and $n\ge1$, put
\[
 S_n\phi(x)=\sum_{j=0}^{n-1}\phi(f^j x).
\]
For $Z\subset X$, define
\[
 \Lambda(n,\varepsilon,Z,\phi)
 =\sup\left\{
      \sum_{y\in E}\exp(S_n\phi(y)):
      E\subset Z\text{ is }(n,\varepsilon)\text{-separated}
   \right\}.
\]

\begin{definition}
	For $\phi\in C(X)$ and $\varepsilon>0$, define
	\[
	P_m(f,\phi,\varepsilon)
	=\limsup_{n\to\infty}\frac{1}{n}\log\sup_{x\in X}
	\Lambda(n,\varepsilon,f^{-n}x,\phi).
	\]
	Then set
	\[
	P_m(f,\phi)=\lim_{\varepsilon\to0}P_m(f,\phi,\varepsilon).
	\]
	Similarly, define
	\[
	P_p(f,\phi)=\sup_{x\in X}\lim_{\varepsilon\to0}
	\limsup_{n\to\infty}\frac{1}{n}\log
	\Lambda(n,\varepsilon,f^{-n}x,\phi).
	\]
	When $\phi=0$, these reduce to the $m$-type and $p$-type pointwise topological preimage entropies $h_m(f)$ and $h_p(f)$, respectively.
\end{definition}

We first give a useful fact.

\begin{lemma}\label{lem:invariance lemma}
Let $U$ be any open neighbourhood of $\mathcal M_f(X)$ in $\mathcal M(X)$. Then there exists $N\ge1$ such that
\[
 \mathcal E_n(y)\in U
 \qquad\text{for every }y\in X\text{ and every }n\ge N.
\]
\end{lemma}

\begin{proof}
If not, there exist $n_k\to\infty$ and $y_k\in X$ with $\mathcal E_{n_k}(y_k)\notin U$. By compactness, a subsequence converges to some $\mu\in\mathcal M(X)$. Clearly $\mu\in\mathcal M_f(X)$, so $\mu\in U$, contradicting $\mathcal E_{n_k}(y_k)\notin U$ for all $k$.
\end{proof}

Now we state and prove the main result of this section. Write $I_\phi(\mu)=\int\phi\,d\mu$.

\begin{theorem}\label{thm:variational principle for preimage pressure}
Let $f:X\to X$ be continuous on a compact metric space and let $\phi\in C(X)$. Then
\[
 P_m(f,\phi)
 =\sup_{\mu\in\mathcal M_f(X)}
 \left\{h^*_{m,\mu}(f)+\int\phi\,d\mu\right\},\quad
 P_p(f,\phi)
 =\sup_{\mu\in\mathcal M_f(X)}
 \left\{h^*_{p,\mu}(f)+\int\phi\,d\mu\right\}.
\]
\end{theorem}

\begin{proof}
\emph{Step 1.}
Fix $\mu\in\mathcal M_f(X)$, $\varepsilon>0$, and $\eta>0$. Since $\phi$ is continuous,
\[
 V_\eta(\mu)=\Bigl\{\lambda\in\mathcal M(X):\Bigl|\int\phi\,d\lambda-\int\phi\,d\mu\Bigr|<\eta\Bigr\}
\]
is a neighbourhood of $\mu$. Let $G\ni\mu$ be any neighbourhood contained in $V_\eta(\mu)$. If $y\in P_n(x,G)$, then $\mathcal E_n(y)\in G\subset V_\eta(\mu) $. So
\[
\frac{1}{n}S_n\phi(y)=\int\phi\,d\mathcal E_n(y)\ge I_\phi(\mu)-\eta.
\]
Hence for any $(n,\varepsilon)$-separated set $E\subset P_n(x,G)$,
\[
\sum_{y\in E}\exp(S_n\phi(y))\ge \#E\cdot\exp\bigl(n(I_\phi(\mu)-\eta)\bigr).
\]
Consequently,
\[
P_m(f,\phi,\varepsilon)\ge h^*_{m,\mu}(f,\varepsilon;G)+I_\phi(\mu)-\eta.
\]
This gives
\[
P_m(f,\phi)\ge h^*_{m,\mu}(f)+I_\phi(\mu).
\]
Since $\mu$ is arbitrary,
\[
P_m(f,\phi)\ge\sup_{\mu\in\mathcal M_f(X)}\bigl\{h^*_{m,\mu}(f)+\int\phi\,d\mu\bigr\}.
\]

\emph{Step 2.}
Fix $\varepsilon>0$ and $\eta>0$. For each $\mu\in\mathcal M_f(X)$, choose an open neighbourhood $W_\mu$ such that
\[
 h^*_{m,\mu}(f,\varepsilon;W_\mu)
 \le h^*_{m,\mu}(f,\varepsilon)+\eta,
\]
and
\[
 \left|\int\phi\,d\lambda-\int\phi\,d\mu\right|<\eta
 \qquad\text{for every }\lambda\in W_\mu.
\]
By compactness, finitely many such neighbourhoods cover $\mathcal M_f(X)$:
\[
\mathcal M_f(X)\subset W_{\mu_1}\cup\cdots\cup W_{\mu_r}.
\]
Set $U=\bigcup_{i=1}^r W_{\mu_i}$. By Lemma~\ref{lem:invariance lemma}, there exists $N\ge1$ such that $\mathcal E_n(y)\in U$ for all $n\ge N$ and $y\in X$.

For $n\ge N$ and $x\in X$, let $E\subset f^{-n}x$ be an $(n,\varepsilon)$-separated set. Partition $E$ into $E_1,\dots,E_r$ where $E_i=\{y\in E:\mathcal E_n(y)\in W_{\mu_i}\}$, choosing the smallest such index if there is more than one. This gives a disjoint decomposition. Then $E_i\subset P_n(x,W_{\mu_i})$, and for $y\in E_i$,
\[
S_n\phi(y)=n\int\phi\,d\mathcal E_n(y)\le n\bigl(I_\phi(\mu_i)+\eta\bigr).
\]
Thus
\[
\sum_{y\in E}\exp(S_n\phi(y))\le\sum_{i=1}^r\exp\bigl(n(I_\phi(\mu_i)+\eta)\bigr)\,\#E_i.
\]
Since $E_i$ is $(n,\varepsilon)$-separated and contained in $P_n(x,W_{\mu_i})$,
\[
\#E_i\le\sup_{z\in X}s\bigl(n,\varepsilon,P_n(z,W_{\mu_i})\bigr).
\]
Therefore,
\[
P_m(f,\phi,\varepsilon)\le\max_{1\le i\le r}\bigl(I_\phi(\mu_i)+\eta+h^*_{m,\mu_i}(f,\varepsilon;W_{\mu_i})\bigr).
\]
By the choice of $W_{\mu_i}$,
\[
P_m(f,\phi,\varepsilon)\le\max_{1\le i\le r}\bigl(I_\phi(\mu_i)+h^*_{m,\mu_i}(f,\varepsilon)+2\eta\bigr).
\]
Letting $\eta\to0$ and then $\varepsilon\to0$ gives
\[
P_m(f,\phi)\le\sup_{\mu\in\mathcal M_f(X)}\bigl\{h^*_{m,\mu}(f)+\int\phi\,d\mu\bigr\}.
\]

Combining both bounds completes the proof of the first variational principle; the second one follows similarly.
\end{proof}

\begin{remark}
	In \cite{ChengNewhouse2005}, Cheng and Newhouse introduced another version of topological preimage entropy:
	\[
	h_{\mathrm{pre}}(f)=\lim_{\varepsilon\to0}\limsup_{n\to\infty}
	\frac1n\log \sup_{x\in X,\,k\ge n}s(n,\varepsilon,f^{-k}x).
	\]
	For a continuous potential $\phi\in C(X)$, this entropy was generalized to the preimage pressure in \cite{ZengYanZhang2007}:
	\[
	P_{\mathrm{pre}}(f,\phi)
	=\lim_{\varepsilon\to0}\limsup_{n\to\infty}\frac{1}{n}\log\sup_{x\in X,\,k\ge n}
	\Lambda(n,\varepsilon,f^{-k}x,\phi).
	\]
	By introducing
	\[
	h_{\mathrm{pre},\mu}^{*}(f)=\lim_{\varepsilon\to0}\inf_{G\ni\mu}\limsup_{n\to\infty}\frac1n\log\sup_{x\in X,\,k\ge n}
	s\bigl(n,\varepsilon,f^{-k}x\cap\{y\in X:\mathcal{E}_n(y)\in G\}\bigr),
	\]
	we can also prove, by a similar approach, that
	\[
	P_{\mathrm{pre}}(f,\phi)
	=\sup_{\mu\in\mathcal{M}_{f}(X)}
	\left\{h_{\mathrm{pre},\mu}^{*}(f)+\int\phi\,d\mu\right\}.
	\]
	In particular, when $\phi=0$, this reduces to a variational principle linking $h_{\mathrm{pre}}(f)$ and $h_{\mathrm{pre},\mu}^{*}(f)$.
\end{remark}

\section{Uniform separation of preimages does not imply forward expansiveness}\label{sec:an example}

First, we give a non-invertible continuous map which has uniform separation of preimages but is not forward expansive.

\begin{example}
Let $\mathbb T=\mathbb R/\mathbb Z$ be the circle with its usual metric
\[
d_{\mathbb T}(a,b)=\min_{m\in\mathbb Z}|a-b-m|.
\]
On the two-dimensional torus $\mathbb T^2$, we consider the product metric
\[
d_{\mathbb T^2}\bigl((\theta,\varphi),(\theta',\varphi')\bigr)
=\max\{d_{\mathbb T}(\theta,\theta'),d_{\mathbb T}(\varphi,\varphi')\}.
\]
Define $F:\mathbb T^2\to\mathbb T^2$ by
\[
F(\theta,\varphi)=(2\theta \bmod 1,\varphi).
\]
Then the continuous map $F$ is non-invertible and has uniform separation of preimages, yet it is not forward expansive.
\end{example}

\begin{proof}

\emph{Step 1.} The map $F$ is non-invertible.

For every $(\theta,\varphi)\in\mathbb T^2$, the two points
\[
\left(\frac{\theta}{2},\varphi\right),\qquad
\left(\frac{\theta+1}{2},\varphi\right)
\]
are distinct and both are mapped to $(\theta,\varphi)$ by $F$. Hence $F$ is not invertible.

\emph{Step 2.} The map $F$ has uniform separation of preimages.

Suppose $F(\theta,\varphi)=F(\theta',\varphi')$. Then we have
$\varphi=\varphi'$ and $2\theta=2\theta'\pmod 1$.
Therefore either $\theta=\theta'$, or $d_{\mathbb T}(\theta,\theta')=\frac12$. Consequently, if
\[
F(\theta,\varphi)=F(\theta',\varphi')
\quad\text{and}\quad
(\theta,\varphi)\neq(\theta',\varphi'),
\]
then necessarily
\[
d_{\mathbb T^2}\bigl((\theta,\varphi),(\theta',\varphi')\bigr)=\frac12.
\]
Thus $F$ has uniform separation of preimages with any separation constant
$0<\varepsilon_0\leq \frac12$.

\emph{Step 3.} The map $F$ is not forward expansive.

Let $\delta>0$ be arbitrary. Choose $\eta\in\mathbb T$ such that
\[
0<d_{\mathbb T}(0,\eta)<\delta.
\]
Set $x=(0,0)$ and $y=(0,\eta)$. Then $x\neq y$. However, for every $n\geq 0$,
\[
F^n(x)=(0,0),\qquad
F^n(y)=(0,\eta).
\]
Hence
\[
d_{\mathbb T^2}(F^n x,F^n y)=d_{\mathbb T}(0,\eta)<\delta
\qquad\forall n\geq 0.
\]
Therefore $F$ is not forward expansive.
\end{proof}

\end{document}